\documentclass{article}
\usepackage{amsfonts}
\usepackage{amsmath}

\newtheorem{theorem}{Theorem}

\newtheorem{corollary}[theorem]{Corollary}

\newtheorem{definition}[theorem]{Definition}

\newtheorem{lemma}[theorem]{Lemma}

\newtheorem{remark}[theorem]{Remark}

\newenvironment{proof}[1][Proof]{\noindent\textbf{#1.} }{\ \rule{0.5em}{0.5em}}
\input{tcilatex}
\begin{document}

\title{\textbf{Growth of Solutions of Linear Difference Equations with
Meromorphic Coefficients in Terms of Iterated }$p-\phi $ \textbf{Order}}
\author{Anirban Bandyopadhyay$^{1}$, Chinmay Ghosh$^{2}$, Sanjib Kumar Datta$%
^{3}$ \\
$^{1}$Gopalpur Primary School\\
Murshidabad, 742304\\
West Bengal, India \\
anirbanbanerjee010@gmail.com \\
$^{2}$Department of Mathematics\\
Kazi Nazrul University\\
Nazrul Road, P.O.- Kalla C.H.\\
Asansol-713340, West Bengal, India\\
chinmayarp@gmail.com\\
$^{2}$Department of Mathematics\\
University of Kalyani\\
P.O.-Kalyani, Dist-Nadia, PIN-741235,\\
West Bengal, India \\
sanjibdatta05@gmail.com\\
}
\date{}
\maketitle

\begin{abstract}
In this article we have studied complex linear homogeneous difference
equations where the coefficients are meromorphic functions, having finite
iterated $p-\phi $ order. We have made some estimations on the growth of its
nontrivial solutions. Also we have extended some of the previous results of
Zhou-Zheng (2017), Bela\"{\i}di-Benkarouba (2019) and Bela\"{\i}di-Bellaama
(2021).

\textbf{AMS Subject Classification }(2010) : 30D35, 39A06, 39B32.\newline
\textbf{Keywords and phrases}: linear difference equation, meromorphic
solution, iterated $p-\phi $ order, iterated $p-\phi $ type, iterated $%
p-\phi $ exponent of convergence.
\end{abstract}

\section{Introduction}

Consider the complex difference equations 
\begin{multline}
A_{k}(z)f(z+c_{k})+A_{k-1}(z)f(z+c_{k-1})+\cdots  \label{1h} \\
+A_{1}(z)f(z+c_{1})+A_{0}(z)f(z)=0
\end{multline}%
and%
\begin{multline}
A_{k}(z)f(z+c_{k})+A_{k-1}(z)f(z+c_{k-1})+\cdots  \label{1nh} \\
+A_{1}(z)f(z+c_{1})+A_{0}(z)f(z)=A(z)
\end{multline}%
where the coefficients $A_{0},A_{1},\ldots ,A_{k}$ $(A_{k}\neq 0)$ and\ $A$ $%
(A\neq 0)$ are meromorphic and $c_{k},c_{k-1},\ldots ,c_{1}\in 
\mathbb{C}
,$ the set of finite complex numbers, are all distinct. Recently, study on (%
\ref{1h}) and (\ref{1nh}) have become a topic of great interests from the
perspective of Nevanlinna's theory\cite{hay,gol}. In 2006 Halburd-Korhonen 
\cite{hal} and in 2008 Chiang-Feng \cite{chi} independently established some
results which are difference analogues of the classical logarithmic
derivative lemma. After that, using those results, many authors (\cite{lai}, 
\cite{zhe}, \cite{liu}, \cite{lat} etc.) studied the growth properties of
the solutions of linear difference equations. In $2017$and in $2019$
respectively Zhou, Zheng \cite{zho} and Bela\"{\i}di, Benkarouba \cite{bel}
investigated higher order linear difference equations and proved some
results estimating the growth of the solutions of those equations using
iterated order \cite{kin} and iterated type \cite{cao}. Again in $2021$, Bela%
\"{\i}di and Bellaama \cite{bel2} used iterated lower order \cite{hu} and
iterated lower type \cite{hu} to improve some of the previous theorems.

In $2009,$ Chyzhykov et. al. \cite{chy} introduced $\phi -$order for a
meromorphic function in the unit disc. Also in $2014,$ Shen et. al. \cite%
{she} developed the concept of $\phi -$order of a meromorphic function in
whole complex plane and applied it to the second order linear differential
equation. The $\phi -$order is based on the function $\phi $ which is
defined as: $\phi :[0,\infty )\longrightarrow (0,\infty )$ is a
nondecreasing, unbounded function which satisfies following two conditions:

\begin{tabular}{ll}
$i)$ & There exists $\,k_{\phi }\in 
\mathbb{N}
,$ such that $\lim\limits_{r\rightarrow \infty }\frac{\log _{p+1}r}{\log
\phi (r)}{\small =0}$ for all $p\geq k_{\phi }{\small ;}$ \\ 
$ii)$ & $\lim\limits_{r\rightarrow \infty }\frac{\log \phi (\alpha r)}{\log
\phi (r)}=1,$ for $\alpha >1.$%
\end{tabular}

In this paper, we first introduce the iterated $p-\phi $ order, the iterated 
$p-\phi $ type and the iterated $p-\phi $ exponent of convergence of a
meromorphic function. Since the concept of iterated $p-\phi $ order is more
generalisation than that of iterated order, we are motivated to improvise
some of the results of Zhou-Zheng \cite{zho}, Bela\"{\i}di-Benkarouba \cite%
{bel} and Bela\"{\i}di-Bellaama \cite{bel2}. Here we obtain our main results
for (\ref{1h}) or (\ref{1nh}), when there may or may not exist exactly one
coefficient having maximal iterated $p-\phi $ order (or iterated lower $%
p-\phi $ order).

Let us define for all $r>0,$%
\begin{equation*}
\log ^{+}r=\max \left( 0,~\log r\right) .
\end{equation*}%
We denote for all $r\in \mathbb{R},$ 
\begin{equation*}
\exp _{1}r=e^{r},~\exp _{n+1}r=\exp (\exp _{n}r),~n\in \mathbb{N}.
\end{equation*}%
Also we denote for all sufficiently large values of $r,$ 
\begin{equation*}
\log _{1}r=\log ^{+}r,~\log _{n+1}r=\log ^{+}(\log _{n}r),~n\in \mathbb{N}.
\end{equation*}%
Further we denote%
\begin{equation*}
\exp _{0}r=\log _{0}r=r,~\exp _{-1}r=\log ^{+}r,~\log _{-1}r=e^{r}.
\end{equation*}%
Moreover, in this paper, the linear measure of a set $E\subset (0,+\infty )$
is defined by 
\begin{equation*}
m\left( E\right) =\int_{0}^{+\infty }\chi _{E}dt
\end{equation*}%
and the logarithmic measure of a set $F\subset (1,+\infty )$ is defined by 
\begin{equation*}
lm\left( F\right) =\int_{1}^{+\infty }\frac{\chi _{F}}{t}dt,
\end{equation*}%
where $\chi _{G}(t)$ is the characteristics function on the set $G$.

\section{Basic definitions and lemmas}

First we give the following definitions to indicate the rate of growth of
meromorphic functions.

\begin{definition}[Iterated $p-\protect\phi $ order]
Let $p\in \mathbb{N}$. Then the iterated $p-\phi $ order, $\rho _{p}(f,\phi
) $ of a meromorphic function $f(z)$ is defined by 
\begin{equation*}
\rho _{p}(f,\phi )=\limsup_{r\rightarrow \infty }\frac{\log _{p}T(r,f)}{\log
\phi (r)},
\end{equation*}%
and when $f$ is an entire function, we define%
\begin{equation*}
\rho _{p}(f,\phi )=\limsup_{r\rightarrow \infty }\frac{\log _{p+1}M_{f}(r)}{%
\log \phi (r)}.
\end{equation*}
\end{definition}

\begin{definition}[Finiteness degree of the $\protect\phi $ order]
The finiteness degree of the $\phi $ order of a meromorphic function $f(z)$
is defined by

$i_{\rho }(f,\phi )=\left\{ 
\begin{tabular}{ll}
$0,$ & $\text{if }\rho _{1}(f,\phi )=\infty \text{,}$ \\ 
&  \\ 
$\min \{j\in 
\mathbb{N}
:\rho _{j}(f,\phi )<\infty \},$ & $\text{if there exist some }j\in 
\mathbb{N}
\text{ with }$ \\ 
& $\rho _{j}(f,\phi )<\infty \text{ exists,}$ \\ 
&  \\ 
$\infty ,$ & $\text{if }\rho _{j}(f,\phi )=\infty \text{ for all }j\in 
\mathbb{N}
.$%
\end{tabular}%
\right. $
\end{definition}

\begin{definition}[Iterated lower $p-\protect\phi $ order]
Let $p\in \mathbb{N}$. Then the iterated lower $p-\phi $ order, $\mu
_{p}(f,\phi )$ of a meromorphic function $f(z)$ is defined by 
\begin{equation*}
\mu _{p}(f,\phi )=\liminf_{r\rightarrow \infty }\frac{\log _{p}T(r,f)}{\log
\phi (r)},
\end{equation*}%
and when $f$ is an entire function, we define%
\begin{equation*}
\mu _{p}(f,\phi )=\liminf_{r\rightarrow \infty }\frac{\log _{p+1}M_{f}(r)}{%
\log \phi (r)}.
\end{equation*}
\end{definition}

\begin{remark}
Let $f$ be a meromorphic function with $i_{\rho }(f,\phi )=p$ and the
iterated $p-\phi $ order of $f$ be $\rho _{p}(f,\phi ).$ Then it\ follows
from the definition that, for any $\epsilon >0,$%
\begin{equation*}
T(r,f)=O\left( \exp _{p-1}\{\phi (r)\}^{\rho _{p}(f,\phi )+\epsilon \text{ }%
}\right) .
\end{equation*}
\end{remark}

\begin{remark}
Let $f$ be a meromorphic function such that, $T(r,f)\leq \exp _{p}\{\phi
(r)\}^{\alpha }$ for some $\alpha >0$ and $p\in 
\mathbb{N}
,$ then we can say%
\begin{equation*}
i_{\rho }(f,\phi )\leq p+1\text{ and }\rho _{p+1}(f,\phi )\leq \alpha .
\end{equation*}
\end{remark}

\begin{definition}[Iterated $p-\protect\phi $ type]
Let $p\in 
\mathbb{N}
.$ Then the iterated $p-\phi $ type, $\tau _{p}(f,\phi )$ of a meromorphic
function $f$ with nonzero finite iterated $p-\phi $ order $\rho _{p}(f,\phi
) $ is defined by%
\begin{equation*}
\tau _{p}(f,\phi )=\limsup_{r\rightarrow \infty }\frac{\log _{p-1}T(r,f)}{%
\left\{ \phi (r)\right\} ^{\rho _{p}(f,\phi )}},
\end{equation*}%
and when $f$ is an entire function, then for $p\geq 2$ we define%
\begin{equation*}
\tau _{p}(f,\phi )=\limsup_{r\rightarrow \infty }\frac{\log _{p}M_{f}(r)}{%
\left\{ \phi (r)\right\} ^{\rho _{p}(f,\phi )}}.
\end{equation*}
\end{definition}

\begin{definition}[Iterated lower $p-\protect\phi $ type]
Let $p\in 
\mathbb{N}
.$ Then the iterated lower $p-\phi $ type, $t_{p}(f,\phi )$ of a meromorphic
function $f$ having non-zero finite iterated lower $p-\phi $ order $\mu
_{p}(f,\phi )$ is defined by%
\begin{equation*}
t_{p}(f,\phi )=\liminf_{r\rightarrow \infty }\frac{\log _{p-1}T(r,f)}{%
\left\{ \phi (r)\right\} ^{\mu _{p}(f,\phi )}},
\end{equation*}%
and when $f$ is an entire function, then for $p\geq 2$ we define%
\begin{equation*}
t_{p}(f,\phi )=\liminf_{r\rightarrow \infty }\frac{\log _{p}M_{f}(r)}{%
\left\{ \phi (r)\right\} ^{\mu _{p}(f,\phi )}}.
\end{equation*}
\end{definition}

\begin{remark}
Let $f$ be a meromorphic function with $i_{\rho }(f,\phi )=p$. Also let the
iterated $p-\phi $ order and the iterated $p-\phi $ type of $f$ be $\rho
_{p}(f,\phi )$ and $\tau _{p}(f,\phi )$ respectively. Then it\ follows from
the definition that, for any $\epsilon >0,$%
\begin{equation*}
T(r,f)=O\left( \exp _{p-1}\left[ \left\{ \tau _{p}(f,\phi )+\epsilon
\right\} \{\phi (r)\}^{\rho _{p}(f,\phi )\text{ }}\right] \right) .
\end{equation*}
\end{remark}

\begin{definition}[Iterated $p-\protect\phi $ exponent of convergence]
Let $p\in 
\mathbb{N}
.$ Then the iterated $p-\phi $ exponent of convergence of the sequence of
poles of a meromorphic function $f$ is defined by%
\begin{equation*}
\lambda _{p}\left( \frac{1}{f},\phi \right) =\limsup_{r\rightarrow \infty }%
\frac{\log _{p}N(r,f)}{\log _{q}\left\{ \phi (r)\right\} }.
\end{equation*}
\end{definition}

Now we give some lemmas which will use to prove our main results.

\begin{lemma}
\label{L1}\cite{hal}Suppose $z_{1}\in \mathbb{C},$ $\delta <1~$and $%
\varepsilon >0.$ Then for any nonconstant meromorphic function $f$, there
exists a set $F\subset (1,+\infty )$ with $lm(F)<\infty $ such that 
\begin{equation*}
m\left( r,\frac{f\left( z+z_{1}\right) }{f\left( z\right) }\right) =o\left( 
\frac{\left\{ T\left( r+\left\vert z_{1}\right\vert ,\text{ }f\right)
\right\} ^{1+\varepsilon }}{r^{\delta }}\right) ,
\end{equation*}%
for all $r\notin F.$
\end{lemma}

\begin{lemma}
\label{L2}\cite{gol} Let $z_{1},$ $z_{2}$ be nonzero complex constants. Then
for any nonconstant meromorphic function $f$, we have%
\begin{equation*}
\left( 1+o(1)\right) T\left( r-\left\vert z_{1}\right\vert ,\text{ }f\right)
\leq T\left( r,\text{ }f\left( z+z_{1}\right) \right) \leq \left(
1+o(1)\right) T\left( r+\left\vert z_{1}\right\vert ,\text{ }f\right) ,
\end{equation*}%
for sufficiently large values of $r.$

Consequently we have, $\rho _{p}\left( f\left( z+z_{1}\right) ,\phi \right)
=\rho _{p}\left( f,\phi \right) $ and $\mu _{p}\left( f\left( z+z_{1}\right)
,\phi \right) =\mu _{p}\left( f,\phi \right) .$
\end{lemma}

\begin{lemma}
\label{L3}\cite{hal}Suppose $z_{1},z_{2}\in \mathbb{C~}(z_{1}\neq z_{2}),$ $%
\delta <1,$ $\varepsilon >0.$ Then, for any nonconstant meromorphic function 
$f$, there exists a set $F\subset (1,+\infty )$ with $lm(F)<\infty $ such
that 
\begin{equation*}
m\left( r,\frac{f\left( z+z_{1}\right) }{f\left( z+z_{2}\right) }\right)
=o\left( \frac{\left\{ T\left( r+\left\vert z_{1}-z_{2}\right\vert
+\left\vert z_{2}\right\vert ,\text{ }f\right) \right\} ^{1+\varepsilon }}{%
r^{\delta }}\right) ,
\end{equation*}%
for all $r\notin F.$
\end{lemma}

\begin{lemma}
\label{L4a}Let $f$ be a nonconstant meromorphic function with nonzero finite
iterated $p-\phi $ order, $\rho _{p}\left( f,\phi \right) .$ Then for any
given $\varepsilon >0,$ there exists a set $F\subset (1,+\infty )$ with $%
lm(F)=+\infty $ such that 
\begin{equation*}
\log _{p-1}T\left( r,\text{ }f\right) >\left\{ \phi (r)\right\} ^{\rho
_{p}\left( f,\phi \right) -\varepsilon },
\end{equation*}%
for $r\in F.$
\end{lemma}

\begin{proof}
By the definition of iterated $p-\phi $ order, there exists an increasing
sequence $\{r_{n}\}~(r_{n}\rightarrow \infty )$ satisfying $(1+\frac{1}{n}%
)r_{n}<r_{n+1}~\ $and%
\begin{equation*}
\lim_{r_{n}\rightarrow \infty }\frac{\log _{p}T(r_{n},f)}{\log \phi (r_{n})}%
=\rho _{p}(f,\phi ).
\end{equation*}

Then for given $\varepsilon >0$\ there exists an integer $N_{0}$ such that
for all $n>N_{0},$%
\begin{equation*}
\log _{p-1}T\left( r_{n},\text{ }f\right) >\left\{ \phi (r_{n})\right\}
^{\rho _{p}\left( f,\phi \right) -\frac{\varepsilon }{2}}
\end{equation*}

Then for $r\in \left[ r_{n},(1+\frac{1}{n})r_{n}\right] $ and for $n>N_{0}$%
\begin{eqnarray*}
\log _{p-1}T\left( r,\text{ }f\right) &>&\log _{p-1}T\left( r_{n},\text{ }%
f\right) >\left\{ \phi (r_{n})\right\} ^{\rho _{p}\left( f,\phi \right) -%
\frac{\varepsilon }{2}} \\
&>&\left\{ \phi (\frac{n}{n+1}r)\right\} ^{\rho _{p}\left( f,\phi \right) -%
\frac{\varepsilon }{2}}>\left\{ \phi (r)\right\} ^{\rho _{p}\left( f,\phi
\right) -\varepsilon }.
\end{eqnarray*}

Set $F=\cup _{n=N_{0}}^{\infty }\left[ r_{n},(1+\frac{1}{n})r_{n}\right] ,$
then%
\begin{eqnarray*}
lm(F) &=&\sum_{n=N_{0}}^{\infty }lm\left( \left[ r_{n},(1+\frac{1}{n})r_{n}%
\right] \right) =\sum_{n=N_{0}}^{\infty }\int_{1}^{+\infty }\frac{\chi _{%
\left[ r_{n},(1+\frac{1}{n})r_{n}\right] }}{t}dt \\
&=&\sum_{n=N_{0}}^{\infty }\int_{r_{n}}^{(1+\frac{1}{n})r_{n}}\frac{1}{t}%
dt=\sum_{n=N_{0}}^{\infty }\log (1+\frac{1}{n})=+\infty .
\end{eqnarray*}

Thus Lemma \ref{L4a} is proved.
\end{proof}

\begin{lemma}
\label{L4b}Let $f$ be a nonconstant meromorphic function with nonzero finite
iterated $p-\phi $ order, $\rho _{p}\left( f,\phi \right) $ and nonzero
finite iterated $p-\phi $ type, $\tau _{p}\left( f,\phi \right) .$ Then for
any given $\varepsilon >0,$ there exists a set $F\subset (1,+\infty )$ with $%
lm(F)=+\infty $ such that 
\begin{equation*}
\log _{p-1}T\left( r,\text{ }f\right) >\left( \tau _{p}\left( f,\phi \right)
-\varepsilon \right) \left\{ \phi (r)\right\} ^{\rho _{p}\left( f,\phi
\right) }.
\end{equation*}%
for $r\in F.$
\end{lemma}

\begin{proof}
By the definition of iterated $p-\phi $ order and iterated $p-\phi $ type,
there exists an increasing sequence $\{r_{n}\}~(r_{n}\rightarrow \infty )$
satisfying $(1+\frac{1}{n})r_{n}<r_{n+1}~\ $and%
\begin{equation*}
\limsup_{r\rightarrow \infty }\frac{\log _{p-1}T(r_{n},f)}{\left\{ \phi
(r_{n})\right\} ^{\rho _{p}(f,\phi )}}=\tau _{p}(f,\phi ).
\end{equation*}

Then for given $\varepsilon >0$\ there exists an integer $N_{0}$ such that
for all $n>N_{0},$%
\begin{equation*}
\log _{p-1}T\left( r_{n},\text{ }f\right) >\left( \tau _{p}\left( f,\phi
\right) -\frac{\varepsilon }{2}\right) \left\{ \phi (r_{n})\right\} ^{\rho
_{p}\left( f,\phi \right) }
\end{equation*}

Then for $r\in \left[ r_{n},(1+\frac{1}{n})r_{n}\right] $ and for $n>N_{0}$%
\begin{gather*}
\log _{p-1}T\left( r,\text{ }f\right) >\log _{p-1}T\left( r_{n},\text{ }%
f\right) >\left( \tau _{p}\left( f,\phi \right) -\frac{\varepsilon }{2}%
\right) \left\{ \phi (r_{n})\right\} ^{\rho _{p}\left( f,\phi \right) } \\
>\left( \tau _{p}\left( f,\phi \right) -\frac{\varepsilon }{2}\right)
\left\{ \phi (\frac{n}{n+1}r)\right\} ^{\rho _{p}\left( f,\phi \right)
}>\left( \tau _{p}\left( f,\phi \right) -\varepsilon \right) \left\{ \phi
(r)\right\} ^{\rho _{p}\left( f,\phi \right) }.
\end{gather*}

Set $F=\cup _{n=N_{0}}^{\infty }\left[ r_{n},(1+\frac{1}{n})r_{n}\right] ,$
then%
\begin{eqnarray*}
lm(F) &=&\sum_{n=N_{0}}^{\infty }lm\left( \left[ r_{n},(1+\frac{1}{n})r_{n}%
\right] \right) =\sum_{n=N_{0}}^{\infty }\int_{1}^{+\infty }\frac{\chi _{%
\left[ r_{n},(1+\frac{1}{n})r_{n}\right] }}{t}dt \\
&=&\sum_{n=N_{0}}^{\infty }\int_{r_{n}}^{(1+\frac{1}{n})r_{n}}\frac{1}{t}%
dt=\sum_{n=N_{0}}^{\infty }\log (1+\frac{1}{n})=+\infty .
\end{eqnarray*}

Thus Lemma \ref{L4b} is proved.
\end{proof}

\begin{lemma}
\label{L4c}Let $f$ be a nonconstant meromorphic function with nonzero finite
iterated lower $p-\phi $ order, $\mu _{p}\left( f,\phi \right) .$ Then for
any given $\varepsilon >0,$ there exists a set $F\subset (1,+\infty )$ with $%
lm(F)=+\infty $ such that 
\begin{equation*}
\log _{p-1}T\left( r,\text{ }f\right) <\left\{ \phi (r)\right\} ^{\mu
_{p}\left( f,\phi \right) +\varepsilon },
\end{equation*}%
for $r\in F.$
\end{lemma}

\begin{proof}
Similar as Lemma \ref{L4a}.
\end{proof}

\begin{lemma}
\label{L4d}Let $f$ be a nonconstant meromorphic function with nonzero finite
iterated lower $p-\phi $ order, $\mu _{p}\left( f,\phi \right) $ and nonzero
finite iterated lower $p-\phi $ type, $t_{p}\left( f,\phi \right) .$ Then
for any given $\varepsilon >0,$ there exists a set $F\subset (1,+\infty )$
with $lm(F)=+\infty $ such that 
\begin{equation*}
\log _{p-1}T\left( r,\text{ }f\right) <\left( t_{p}\left( f,\phi \right)
+\varepsilon \right) \left\{ \phi (r)\right\} ^{\mu _{p}\left( f,\phi
\right) }.
\end{equation*}%
for $r\in F.$
\end{lemma}

\begin{proof}
Similar as Lemma \ref{L4b}.
\end{proof}

\begin{lemma}
\label{L5a}Let $f$ be a nonconstant meromorphic function with $i_{\rho
}(f,\phi )=p~$and $\rho _{p}\left( f,\phi \right) =\rho $. Also let $%
z_{1},z_{2}\in \mathbb{C~}(z_{1}\neq z_{2})$. Then for any $\varepsilon >0$
sufficiently small, we have%
\begin{equation*}
m\left( r,\frac{f\left( z+z_{1}\right) }{f\left( z+z_{2}\right) }\right)
=O\left( \exp _{p-1}\left\{ \phi (r)\right\} ^{\rho +\varepsilon }\right) .
\end{equation*}
\end{lemma}

\begin{proof}
By using similar arguments as in the proof of (\cite{bel}, Lemma 2.9) we get
for some $M>0$,%
\begin{equation*}
m\left( r,\frac{f\left( z+z_{1}\right) }{f\left( z+z_{2}\right) }\right)
\leq M.T\left( 3r,f\right) ,
\end{equation*}

Since $f$ has finite iterated $p-\phi $ order, $\rho _{p}\left( f,\phi
\right) =\rho ,$ we have for any $\varepsilon $%
\begin{equation*}
T\left( r,f\right) \leq \exp _{p-1}\left\{ \phi (r)\right\} ^{\rho +\frac{%
\varepsilon }{2}}
\end{equation*}

Combining above two results we get,%
\begin{eqnarray*}
m\left( r,\frac{f\left( z+z_{1}\right) }{f\left( z+z_{2}\right) }\right)
&\leq &M.T\left( 3r,f\right) \leq M.\exp _{p-1}\left\{ \phi (3r)\right\}
^{\rho +\frac{\varepsilon }{2}} \\
&\leq &M.\exp _{p-1}\left\{ \phi (r)\right\} ^{\rho +\varepsilon }.
\end{eqnarray*}

This completes the proof.
\end{proof}

\begin{lemma}
\label{L5b}Let $f$ be a nonconstant meromorphic function with $i_{\rho
}(f,\phi )=p~$and $\mu _{p}\left( f,\phi \right) =\mu $. Also let $%
z_{1},z_{2}\in \mathbb{C~}(z_{1}\neq z_{2})$. Then for any $\varepsilon >0$
sufficiently small, there exists a set $F\subset (1,+\infty )$ with $%
lm(F)=+\infty $ such that 
\begin{equation*}
m\left( r,\frac{f\left( z+z_{1}\right) }{f\left( z+z_{2}\right) }\right)
=O\left( \exp _{p-1}\left\{ \phi (r)\right\} ^{\mu +\varepsilon }\right) ,
\end{equation*}%
for $r\in F.$
\end{lemma}

\begin{proof}
As in the previous lemma for some $M>0$, we get%
\begin{equation*}
m\left( r,\frac{f\left( z+z_{1}\right) }{f\left( z+z_{2}\right) }\right)
\leq M.T\left( 3r,f\right) ,
\end{equation*}

Since $f$ has finite iterated lower $p-\phi $ order, $\mu _{p}\left( f,\phi
\right) =\mu ,$ we have from Lemma \ref{L4c}, for any $\varepsilon >0$ there
exists a set $F\subset (1,+\infty )$ with $lm(F)=+\infty $ such that 
\begin{equation*}
T\left( r,f\right) \leq \exp _{p-1}\left\{ \phi (r)\right\} ^{\mu +\frac{%
\varepsilon }{2}},
\end{equation*}%
for $r\in F.$

Combining above two results we get,%
\begin{eqnarray*}
m\left( r,\frac{f\left( z+z_{1}\right) }{f\left( z+z_{2}\right) }\right)
&\leq &M.T\left( 3r,f\right) \leq M.\exp _{p-1}\left\{ \phi (3r)\right\}
^{\mu +\frac{\varepsilon }{2}} \\
&\leq &M.\exp _{p-1}\left\{ \phi (r)\right\} ^{\mu +\varepsilon },\text{ }
\end{eqnarray*}%
for $r\in F.$

This completes the proof.
\end{proof}

\begin{lemma}
\label{L5}Let $f$ be a nonconstant meromorphic function with $i_{\rho
}(f,\phi )=p~$and $\rho _{p}\left( f,\phi \right) =\rho $. Suppose $z_{1}\in 
\mathbb{C~}(z_{1}\neq 0)$ and $\varepsilon >0$ be any given constant
sufficiently small. Then there exists a set $F\subset (1,+\infty )$ with $%
lm(F)<\infty $ such that%
\begin{equation*}
\left\vert \frac{f\left( z+z_{1}\right) }{f\left( z\right) }\right\vert \leq
\exp _{p}\left\{ \phi (r)\right\} ^{\rho +\varepsilon },
\end{equation*}%
for $r\notin F\cup \left[ 0,1\right] .$
\end{lemma}

\begin{proof}
By using similar arguments as in the proof of (\cite{bel}, Lemma 2.3), there
exists a set $F\subset (1,+\infty )$ with $lm(F)=+\infty $ such that for
some $B>0,~\lambda >1$, we have%
\begin{equation*}
\left\vert \log \left\vert \frac{f\left( z+z_{1}\right) }{f\left( z\right) }%
\right\vert \right\vert \leq B\left\{ T\left( \lambda r,f\right) \frac{\log
^{\lambda }r}{r}\log T\left( \lambda r,f\right) \right\} ,
\end{equation*}%
for $r\notin F\cup \left[ 0,1\right] .$ Again since, $f$ has finite iterated 
$p-\phi $ order, $\rho _{p}\left( f,\phi \right) =\rho ,$ we have for any
given $\varepsilon >0,$ 
\begin{equation*}
T\left( r,f\right) \leq \exp _{p-1}\left\{ \phi (r)\right\} ^{\rho +\frac{%
\varepsilon }{2}}
\end{equation*}

Combining above two results we get%
\begin{eqnarray*}
\left\vert \log \left\vert \frac{f\left( z+z_{1}\right) }{f\left( z\right) }%
\right\vert \right\vert &\leq &B\exp _{p-1}\left\{ \phi (\lambda r)\right\}
^{\rho +\frac{\varepsilon }{2}}\frac{\log ^{\lambda }r}{r}\exp _{p-2}\left\{
\phi (\lambda r)\right\} ^{\rho +\frac{\varepsilon }{2}} \\
&\leq &\exp _{p-1}\left\{ \phi (r)\right\} ^{\rho +\varepsilon },\text{ for }%
r\notin F\cup \left[ 0,1\right] . \\
\text{Therefore, }\left\vert \frac{f\left( z+z_{1}\right) }{f\left( z\right) 
}\right\vert &\leq &\exp _{p}\left\{ \phi (r)\right\} ^{\rho +\varepsilon },%
\text{ for }r\notin F\cup \left[ 0,1\right] .
\end{eqnarray*}

Thus Lemma \ref{L5} is proved.
\end{proof}

\begin{lemma}
\label{L5c}Let $f$ be a nonconstant meromorphic function with $i_{\rho
}(f,\phi )=p~$and $\rho _{p}\left( f,\phi \right) =\rho $. Suppose $%
z_{1},z_{2}\in \mathbb{C~}(z_{1}\neq z_{2})$ and $\varepsilon >0$ be any
given constant sufficiently small. Then there exists a set $F\subset
(1,+\infty )$ with $lm(F)<\infty $ such that%
\begin{equation*}
\left\vert \frac{f\left( z+z_{1}\right) }{f\left( z+z_{2}\right) }%
\right\vert \leq \exp _{p}\left\{ \phi (r)\right\} ^{\rho +\varepsilon },
\end{equation*}%
for $r\notin F\cup \left[ 0,1\right] .$
\end{lemma}

\begin{proof}
We can write%
\begin{equation*}
\left\vert \frac{f\left( z+z_{1}\right) }{f\left( z+z_{2}\right) }%
\right\vert =\left\vert \frac{f\left( z+z_{2}+z_{1}-z_{2}\right) }{f\left(
z+z_{2}\right) }\right\vert .
\end{equation*}

By Lemma \ref{L2}, we have $\rho _{p}\left( f\left( z+z_{2}\right) ,\phi
\right) =\rho _{p}\left( f,\phi \right) =\rho .$ Again since $%
z_{1}-z_{2}\neq 0,~$by Lemma \ref{L5}, there exists a set $F\subset
(1,+\infty )$ with $lm(F)=+\infty $ such that for given any $\varepsilon >0,$
we get 
\begin{equation*}
\left\vert \frac{f\left( z+z_{2}+z_{1}-z_{2}\right) }{f\left( z+z_{2}\right) 
}\right\vert \leq \exp _{p}\left\{ \phi (r)\right\} ^{\rho +\varepsilon },%
\text{ for }r\notin F\cup \left[ 0,1\right] .
\end{equation*}

Thus Lemma \ref{L5c} is proved.
\end{proof}

\section{Main Results}

We now state and prove our main theorems. In Theorem \ref{T1} and Theorem %
\ref{T2} we consider homogeneous linear difference equation (\ref{1h}) with
exactly one coefficient having maximal iterated $p-\phi $ order and iterated
lower $p-\phi $ order respectively.

\begin{theorem}
\label{T1}Let $A_{j}(z)(j=0,1,\ldots ,k)$ be meromorphic functions and $%
p=\max \{i_{\rho }(A_{j},\phi ):j=0,1,\ldots ,n-1\}$. \bigskip If there
exits an $A_{m}(z)(0\leq m\leq k)$ such that 
\begin{eqnarray*}
\lambda _{p}\left( \frac{1}{A_{m}},\phi \right) &<&\rho _{p}\left(
A_{m},\phi \right) <+\infty , \\
\text{and }\max \{\rho _{p}\left( A_{j},\phi \right) &:&j=0,1,\ldots
,k,~j\neq m\}<\rho _{p}\left( A_{m},\phi \right)
\end{eqnarray*}%
then every solution $f(z)~(f\neq 0)$ of (\ref{1h}) satisfies $i_{\rho
}(f,\phi )\geq p~$and $\rho _{p}\left( f,\phi \right) \geq \rho _{p}\left(
A_{m},\phi \right) .$
\end{theorem}

\begin{proof}
Let $f$ be a meromorphic solution of (\ref{1h}) and put $c_{0}=0$. We divide
(\ref{1h}) by $f(z+c_{m})~$and we get%
\begin{equation}
-A_{m}(z)=\sum_{j=0,j\neq m}^{k}A_{j}(z)\frac{f(z+c_{j})}{f(z+c_{m})}.
\label{1.1}
\end{equation}

Now from Lemma \ref{L3} for any $\varepsilon >0$ we get,%
\begin{equation*}
m\left( r,\frac{f(z+c_{j})}{f(z+c_{m})}\right) =o\left( \frac{\left\{
T(r+3C,f)\right\} ^{1+\varepsilon }}{r^{\delta }}\right) ~,j=0,1,\ldots
,k,~j\neq m,
\end{equation*}%
for$~r\notin F_{1},$ where $lm(F_{1})<+\infty $ and $C=\max \{\left\vert
c_{j}\right\vert :j=0,1,\ldots ,k\}.$

Using the above result, from (\ref{1.1}) we get$,$ 
\begin{eqnarray}
&&T(r,A_{m})  \notag \\
&=&m(r,A_{m})+N(r,A_{m})  \notag \\
&\leq &\sum_{j=0,j\neq m}^{k}m(r,A_{j})+\sum_{j=0,j\neq m}^{k}m\left( r,%
\frac{f(z+c_{j})}{f(z+c_{m})}\right) +N(r,A_{m})+O(1)  \label{1.1a} \\
&\leq &\sum_{j=0,j\neq m}^{k}T(r,A_{j})+o\left( \frac{\left\{
T(r+3C,f)\right\} ^{1+\varepsilon }}{r^{\delta }}\right) +N(r,A_{m})+O(1) 
\notag \\
&\leq &\sum_{j=0,j\neq m}^{k}T(r,A_{j})+\left\{ T(2r,f)\right\}
^{2}+N(r,A_{m})+O(1)  \label{1.2}
\end{eqnarray}%
for $r\notin F_{1}.$

We denote,$~\rho =\rho _{p}\left( A_{m},\phi \right) ,~\rho _{1}=\max \{\rho
_{p}\left( A_{j},\phi \right) :j=0,1,\ldots ,k;~~j\neq m\}$ and $\lambda
=\lambda _{p}\left( \frac{1}{A_{m}},\phi \right) .~$

By Lemma \ref{L4a}, for that $\varepsilon $ there exists a set $F_{2}\subset
(1,+\infty )$ with $lm(F_{2})=+\infty $ such that%
\begin{equation}
T(r,A_{m})>\exp _{p-1}\left\{ \phi (r)\right\} ^{\rho -\varepsilon }
\label{1.3}
\end{equation}%
for $r\in F_{2}$. Also for $~j\neq m,~$for that $\varepsilon ~$we have%
\begin{equation}
T(r,A_{j})\leq \exp _{p-1}\left\{ \phi (r)\right\} ^{\rho _{1}+\varepsilon }
\label{1.4}
\end{equation}%
for sufficiently large $r.$

Again by the definition of $\lambda _{p}\left( \frac{1}{A_{m}},\phi \right)
, $ we have for that $\varepsilon $ and for sufficiently large $r$%
\begin{equation}
N(r,A_{m})\leq \exp _{p-1}\left\{ \phi (r)\right\} ^{\lambda +\varepsilon }.
\label{1.5}
\end{equation}

Now, using the relations (\ref{1.3})-(\ref{1.5}) and choosing $\varepsilon $
such that $0<\varepsilon <\frac{1}{2}\min \left\{ \rho -\rho _{1},\rho
-\lambda \right\} ,$ we have from (\ref{1.2})%
\begin{gather*}
\exp _{p-1}\left\{ \phi (r)\right\} ^{\rho -\varepsilon }\leq k.\exp
_{p-1}\left\{ \phi (r)\right\} ^{\rho _{1}+\varepsilon }+\left\{
T(2r,f)\right\} ^{2} \\
+\exp _{p-1}\left\{ \phi (r)\right\} ^{\lambda +\varepsilon }+O(1) \\
\Rightarrow \left\{ T(2r,f)\right\} ^{2}\geq O\left( \exp _{p-1}\left\{ \phi
(r)\right\} ^{\rho -\varepsilon }\right)
\end{gather*}%
for sufficiently large $r$ and $r\in F_{2}\backslash F_{1}.$

Which implies, $i_{\rho }(f,\phi )\geq p~$and $\rho _{p}\left( f,\phi
\right) \geq \rho =\rho _{p}\left( A_{m},\phi \right) .$
\end{proof}

\begin{theorem}
\label{T2}Let $A_{j}(z)(j=0,1,\ldots ,k)$ be meromorphic functions and $%
p=\max \{i_{\rho }(A_{j},\phi ):j=0,1,\ldots ,n-1\}$. \bigskip If there
exits an $A_{m}(z)(0\leq m\leq k)$ such that 
\begin{eqnarray*}
\lambda _{p}\left( \frac{1}{A_{m}},\phi \right) &<&\mu _{p}\left( A_{m},\phi
\right) <+\infty , \\
\text{and }\max \{\rho _{p}\left( A_{j},\phi \right) &:&j=0,1,\ldots
,k,~j\neq m\}<\mu _{p}\left( A_{m},\phi \right)
\end{eqnarray*}%
then every solution $f(z)~(f\neq 0)$ of (\ref{1h}) satisfies $i_{\rho
}(f,\phi )\geq p~$and $\mu _{p}\left( f,\phi \right) \geq \mu _{p}\left(
A_{m},\phi \right) .$
\end{theorem}

\begin{proof}
Let $\varepsilon >0$\ be any arbitrary constant. As in the previous theorem,
from (\ref{1.2}), we get, 
\begin{equation}
T(r,A_{m})\leq \sum_{j=0,j\neq m}^{k}T(r,A_{j})+\left\{ T(2r,f)\right\}
^{2}+N(r,A_{m})+O(1)  \label{1a.1}
\end{equation}%
for $r\notin F_{1},$ where $lm(F_{1})<+\infty .$

We denote,$~\mu =\mu _{p}\left( A_{m},\phi \right) ,~\rho _{1}=\max \{\rho
_{p}\left( A_{j},\phi \right) :j=0,1,\ldots ,k;~~j\neq m\}$ and $\lambda
=\lambda _{p}\left( \frac{1}{A_{m}},\phi \right) .~$

Then for that $\varepsilon $ and, we have%
\begin{equation}
T(r,A_{m})>\exp _{p-1}\left\{ \phi (r)\right\} ^{\mu -\varepsilon }
\label{1a.2}
\end{equation}%
for sufficiently large $r$. Also for $~j\neq m,~$for that $\varepsilon ~$we
have%
\begin{equation}
T(r,A_{j})\leq \exp _{p-1}\left\{ \phi (r)\right\} ^{\rho _{1}+\varepsilon }
\label{1a.3}
\end{equation}%
for sufficiently large $r.$

Again by the definition of $\lambda _{p}\left( \frac{1}{A_{m}},\phi \right)
, $ we have for that $\varepsilon $ and for sufficiently large $r$%
\begin{equation}
N(r,A_{m})\leq \exp _{p-1}\left\{ \phi (r)\right\} ^{\lambda +\varepsilon }.
\label{1a.4}
\end{equation}

Now, using the relations (\ref{1a.2})-(\ref{1a.4}) and choosing $\varepsilon 
$ such that $0<\varepsilon <\frac{1}{2}\min \left\{ \mu -\rho _{1},\mu
-\lambda \right\} ,$ we have from (\ref{1a.1})%
\begin{gather*}
\exp _{p-1}\left\{ \phi (r)\right\} ^{\mu -\varepsilon }\leq k.\exp
_{p-1}\left\{ \phi (r)\right\} ^{\rho _{1}+\varepsilon }+\left\{
T(2r,f)\right\} ^{2} \\
+\exp _{p-1}\left\{ \phi (r)\right\} ^{\lambda +\varepsilon }+O(1) \\
\Rightarrow \left\{ T(2r,f)\right\} ^{2}\geq O\left( \exp _{p-1}\left\{ \phi
(r)\right\} ^{\mu -\varepsilon }\right)
\end{gather*}%
for $r\notin F_{1}.$

Which implies, $i_{\rho }(f,\phi )\geq p~$and $\mu _{p}\left( f,\phi \right)
\geq \mu =\mu _{p}\left( A_{m},\phi \right) .$
\end{proof}

In Theorem \ref{T3} and Theorem \ref{T4} we consider nonhomogeneous linear
difference equation (\ref{1nh}) which may have more than one coefficient
with the maximal iterated $p-\phi $ order and iterated lower $p-\phi $ order
respectively. For these we need to consider the iterated $p-\phi $ type or
iterated lower $p-\phi $ type among the coefficients having maximal iterated 
$p-\phi $ order.

\begin{theorem}
\label{T3}Let $A_{j}(z)(j=0,1,\ldots ,k)$ and $A(z)$ be meromorphic
functions and $p=\max \{i_{\rho }(A_{j},\phi ):j=0,1,\ldots ,n-1\}$. If
there exits an $A_{m}(z)(0\leq m\leq k)$ such that 
\begin{eqnarray*}
\lambda _{p}\left( \frac{1}{A_{m}},\phi \right) &<&\rho _{p}\left(
A_{m},\phi \right) <\infty , \\
\max \{\rho _{p}\left( A_{j},\phi \right) &:&j=0,1,\ldots ,k,~j\neq m\}\leq
\rho _{p}\left( A_{m},\phi \right) \\
\text{and }\max \{\tau _{p}\left( A_{j},\phi \right) &:&\rho _{p}\left(
A_{j},\phi \right) =\rho _{p}\left( A_{m},\phi \right) ,~j=0,1,\ldots
,k,~j\neq m\}<\tau _{p}\left( A_{m},\phi \right) <\infty ,
\end{eqnarray*}%
then the following cases can be arise.

i)\qquad If $\rho _{p}\left( A,\phi \right) <\rho _{p}\left( A_{m},\phi
\right) ,$ or$~\rho _{p}\left( A,\phi \right) =\rho _{p}\left( A_{m},\phi
\right) $ and $\tau _{p}\left( A,\phi \right) \neq \tau _{p}\left(
A_{m},\phi \right) ,$ then every solution $f(z)~(f\neq 0)$ of (\ref{1nh})
satisfies $i_{\rho }(f,\phi )\geq p~$and $\rho _{p}\left( f,\phi \right)
\geq \rho _{p}\left( A_{m},\phi \right) .$

ii)\qquad If $\rho _{p}\left( A,\phi \right) >\rho _{p}\left( A_{m},\phi
\right) $ then every solution $f(z)~(f\neq 0)$ of (\ref{1nh}) satisfies $%
i_{\rho }(f,\phi )\geq p~$and $\rho _{p}\left( f,\phi \right) \geq \rho
_{p}\left( A,\phi \right) .$
\end{theorem}

\begin{proof}
Let $f$ be a meromorphic solution of (\ref{1nh}) and put $c_{0}=0$. We
divide (\ref{1nh}) by $f(z+c_{m})~$and we get%
\begin{equation}
-A_{m}(z)=\sum_{j=0,j\neq m}^{k}A_{j}(z)\frac{f(z+c_{j})}{f(z+c_{m})}-\frac{%
A(z)}{f(z+c_{m})}.  \label{2.1}
\end{equation}

Now from Lemma \ref{L3} and Lemma \ref{L2}, for any $\varepsilon >0$ we get,%
\begin{eqnarray*}
m\left( r,\frac{f(z+c_{j})}{f(z+c_{m})}\right) &\leq &o\left( \frac{\left\{
T(r+3C,f)\right\} ^{1+\varepsilon }}{r^{\delta }}\right) ~,j=0,1,\ldots
,k,~j\neq m, \\
\text{and }m\left( r,\frac{1}{f(z+c_{m})}\right) &\leq &T\left( r,\frac{1}{%
f(z+c_{m})}\right) \\
&=&T\left( r,f(z+c_{m})\right) +O(1) \\
&\leq &(1+O(1))T\left( r+C,f\right) ,
\end{eqnarray*}%
for$~r\notin F_{1},$ where $lm(F_{1})<+\infty $ and $C=\max \{\left\vert
c_{j}\right\vert :j=0,1,\ldots ,k\}$.

Using above results in (\ref{2.1}) we get, 
\begin{eqnarray}
&&T(r,A_{m})  \notag \\
&=&m(r,A_{m})+N(r,A_{m})  \notag \\
&\leq &\sum_{j=0,j\neq m}^{k}m(r,A_{j})+\sum_{j=0,j\neq m}^{k}m\left( r,%
\frac{f(z+c_{j})}{f(z+c_{m})}\right) +m(r,A)  \notag \\
&&+m\left( r,\frac{1}{f(z+c_{m})}\right) +N(r,A_{m})+O(1)  \notag \\
&\leq &\sum_{j=0,j\neq m}^{k}T(r,A_{j})+o\left( \frac{\left\{
T(r+3C,f)\right\} ^{1+\varepsilon }}{r^{\delta }}\right) +T(r,A)  \notag \\
&&+(1+O(1))T\left( r+C,f\right) +N(r,A_{m})+O(1)  \notag \\
&\leq &\sum_{j=0,j\neq m}^{k}T(r,A_{j})+3\left\{ T(2r,f)\right\}
^{2}+T(r,A)+N(r,A_{m})+O(1)  \label{2.2}
\end{eqnarray}%
for sufficiently large $r$ and $r\notin F_{1}.$

Now denote,$~\rho =\rho _{p}\left( A_{m},\phi \right) ,~\rho _{1}=\max
\{\rho _{p}\left( A_{j},\phi \right) :j=0,1,\ldots ,k,~\rho _{p}\left(
A_{j},\phi \right) <\rho \},~\tau =\tau _{p}\left( A_{m},\phi \right) ,~\tau
_{1}=\max \{\tau _{p}\left( A_{j},\phi \right) :j=0,1,\ldots ,k,~j\neq
m,~\rho _{p}\left( A_{j},\phi \right) =\rho \}$ and $\lambda =\lambda
_{p}\left( \frac{1}{A_{m}},\phi \right) .$

Now from lemma \ref{L4b} for that $\varepsilon ~$there exists a set $%
F_{2}\subset (1,+\infty )$ with $lm(F)=+\infty $ such that%
\begin{equation}
T(r,A_{m})>\exp _{p-1}\left[ (\tau -\varepsilon )\left\{ \phi (r)\right\}
^{\rho }\right]  \label{2.3}
\end{equation}%
for sufficiently large $r$ and $r\in F_{2}$.

Again if $\rho _{p}\left( A_{j},\phi \right) <\rho ,$ then for that $%
\varepsilon ~$we have%
\begin{equation}
T(r,A_{j})\leq \exp _{p-1}\left[ \left\{ \phi (r)\right\} ^{\rho
_{1}+\varepsilon }\right]  \label{2.4}
\end{equation}%
for sufficiently large $r.$

If$~\rho _{p}\left( A_{j},\phi \right) =\rho ,~$ $j\neq m,~$then for that $%
\varepsilon ~$we have%
\begin{equation}
T(r,A_{j})\leq \exp _{p-1}\left[ (\tau _{1}+\varepsilon )\left\{ \phi
(r)\right\} ^{\rho }\right]  \label{2.5}
\end{equation}%
for sufficiently large $r.$

Also by the definition of $\lambda _{p}\left( \frac{1}{A_{m}},\phi \right) ,$
we have for the above $\varepsilon $ and for sufficiently large $r$%
\begin{equation}
N(r,A_{m})\leq \exp _{p-1}\left\{ \phi (r)\right\} ^{\lambda +\varepsilon }.
\label{2.6}
\end{equation}

(i) Case 1: For the first part of the proof, we take $\rho _{p}\left( A,\phi
\right) <\rho ,~$then for that $\varepsilon ~$we have%
\begin{equation}
T(r,A)\leq \exp _{p-1}\left[ \left\{ \phi (r)\right\} ^{\rho _{p}\left(
A,\phi \right) +\varepsilon }\right]  \label{2.7}
\end{equation}%
for sufficiently large $r.$

Now, using the relations (\ref{2.3})-(\ref{2.7}) and choosing $\varepsilon $
such that $0<\varepsilon <\frac{1}{2}\min \left\{ \rho -\rho _{1},\tau -\tau
_{1},\rho -\lambda ,\rho -\rho _{p}\left( A,\phi \right) \right\} ,$ we have
from (\ref{2.2})%
\begin{gather*}
\exp _{p-1}\left[ (\tau -\varepsilon )\left\{ \phi (r)\right\} ^{\rho }%
\right] \leq O\left( \exp _{p-1}\left[ \left\{ \phi (r)\right\} ^{\rho
_{1}+\varepsilon }\right] \right) +O\left( \exp _{p-1}\left[ (\tau
_{1}+\varepsilon )\left\{ \phi (r)\right\} ^{\rho }\right] \right) \\
+3\left\{ T(2r,f)\right\} ^{2}+\exp _{p-1}\left[ \left\{ \phi (r)\right\}
^{\rho _{p}\left( A,\phi \right) +\varepsilon }\right] +\exp _{p-1}\left[
\left\{ \phi (r)\right\} ^{\lambda +\varepsilon }\right] +O(1) \\
\Rightarrow \left\{ T(2r,f)\right\} ^{2}\geq O\left( \exp _{p-1}\left[
\left\{ \phi (r)\right\} ^{\rho +\varepsilon }\right] \right) ,
\end{gather*}%
for sufficiently large $r$ and $r\in F_{2}\backslash F_{1}.$

Which implies, $i_{\rho }(f,\phi )\geq p~$and $\rho _{p}\left( f,\phi
\right) \geq \rho =\rho _{p}\left( A_{m},\phi \right) .$

Case 2: Next we suppose that $\rho _{p}\left( A,\phi \right) =\rho $ and $%
\tau _{p}\left( A,\phi \right) <\tau ,~$then for the above $\varepsilon ~$we
have%
\begin{equation}
T(r,A)\leq \exp _{p-1}\left[ \left\{ \tau _{p}\left( A,\phi \right)
+\varepsilon \right\} \left\{ \phi (r)\right\} ^{\rho }\right]  \label{2.8}
\end{equation}%
for sufficiently large $r.$

Now, using relations (\ref{2.3})-(\ref{2.6}) and (\ref{2.8}) and choosing $%
\varepsilon $ such that $0<\varepsilon <\frac{1}{2}\min \left\{ \rho -\rho
_{1},\tau -\tau _{1},\rho -\lambda ,\tau -\tau _{p}\left( A,\phi \right)
\right\} ,$ we have from (\ref{2.2})%
\begin{gather*}
\exp _{p-1}\left[ (\tau -\varepsilon )\left\{ \phi (r)\right\} ^{\rho }%
\right] \leq O\left( \exp _{p-1}\left[ \left\{ \phi (r)\right\} ^{\rho
_{1}+\varepsilon }\right] \right) +O\exp _{p-1}\left[ (\tau _{1}+\varepsilon
)\left\{ \phi (r)\right\} ^{\rho }\right] \\
+3\left\{ T(2r,f)\right\} ^{2}+\exp _{p-1}\left[ \left\{ \tau _{p}\left(
A,\phi \right) +\varepsilon \right\} \left\{ \phi (r)\right\} ^{\rho }\right]
+\exp _{p-1}\left[ \left\{ \phi (r)\right\} ^{\lambda +\varepsilon }\right]
+O(1) \\
\Rightarrow \left\{ T(2r,f)\right\} ^{2}>O\left( \exp _{p-1}\left[ \left\{
\phi (r)\right\} ^{\rho +\varepsilon }\right] \right)
\end{gather*}%
for sufficiently large $r$ and $r\in F_{2}\backslash F_{1}.$

Which implies, $i_{\rho }(f,\phi )\geq p~$and $\rho _{p}\left( f,\phi
\right) \geq \rho =\rho _{p}\left( A_{m},\phi \right) .$

Case 3: Next we take, $\rho _{p}\left( A,\phi \right) =\rho $ and $\tau
_{p}\left( A,\phi \right) >\tau ,$ then by Lemma \ref{L4b}, and for the
above $\varepsilon ~$we have%
\begin{equation}
T(r,A)>\exp _{p-1}\left[ \left\{ \tau _{p}\left( A,\phi \right) -\varepsilon
\right\} \left\{ \phi (r)\right\} ^{\rho }\right]  \label{2.9}
\end{equation}%
for $r\in F_{3},$ where $lm(F_{3})=+\infty $.

Again by the definition of $\tau _{p}\left( A_{m},\phi \right) $ $,$ we have
for the above $\varepsilon $ and for sufficiently large $r$%
\begin{equation}
T(r,A_{m})\leq \exp _{p-1}\left[ (\tau +\varepsilon )\left\{ \phi
(r)\right\} ^{\rho }\right] .  \label{2.10}
\end{equation}

Now from (\ref{1nh}) and using Lemma \ref{L2} we get%
\begin{gather}
A(z)=\sum_{j=0}^{k}A_{j}(z)f(z+c_{j})  \notag \\
\Rightarrow T(r,A)\leq \sum_{j=0,j\neq
m}^{k}T(r,A_{j})+T(r,A_{m})+\sum_{j=0}^{k}T(r,f(z+c_{j}))+O(1)  \notag \\
\Rightarrow T(r,A)\leq \sum_{j=0,j\neq
m}^{k}T(r,A_{j})+T(r,A_{m})+(k+1)(1+o(1))T(r+C,f)+O(1)  \notag \\
\Rightarrow T(r,A)\leq \sum_{j=0,j\neq
m}^{k}T(r,A_{j})+T(r,A_{m})+2(k+1)T(2r,f)+O(1)  \label{2.11}
\end{gather}%
for sufficiently large $r$

Now, using relations (\ref{2.4}), (\ref{2.5}), (\ref{2.9}), (\ref{2.10}) and
choosing $\varepsilon $ such that $0<\varepsilon <\frac{1}{2}\min \left\{
\rho -\rho _{1},\tau -\tau _{1},\tau _{p}\left( A,\phi \right) -\tau
\right\} ,$ we have from (\ref{2.11}) 
\begin{gather*}
\exp _{p-1}\left[ (\tau _{p}\left( A,\phi \right) -\varepsilon )(\phi
(r))^{\rho }\right] \leq O\left[ \exp _{p-1}\left[ \left\{ \phi (r)\right\}
^{\rho _{1}+\varepsilon }\right] \right] \\
+O\left( \exp _{p-1}\left[ (\tau _{1}+\varepsilon )\left\{ \phi (r)\right\}
^{\rho }\right] \right) +\exp _{p-1}\left[ (\tau +\varepsilon )\left\{ \phi
(r)\right\} ^{\rho }\right] +2(k+1)T(2r,f)+O(1) \\
\Rightarrow T(2r,f)\geq O\left( \exp _{p-1}\left[ \left\{ \phi (r)\right\}
^{\rho +\varepsilon }\right] \right)
\end{gather*}%
for sufficiently large $r$ and $r\in F_{3}.$

It follows that, $i_{\rho }(f,\phi )\geq p~$and $\rho _{p}\left( f,\phi
\right) \geq \rho =\rho _{p}\left( A_{m},\phi \right) .$

(ii) For the second part of the theorem, we take $\rho _{p}\left( A,\phi
\right) >\rho $. Then from lemma \ref{L4a} for the above $\varepsilon ~$we
have%
\begin{equation}
T(r,A)\geq \exp _{p-1}\left[ \left\{ \phi (r)\right\} ^{\rho _{p}\left(
A,\phi \right) -\varepsilon }\right]  \label{2.12}
\end{equation}%
for sufficiently large $r$ with$~r\in F_{4},$ where $lm(F_{4})=+\infty $.

Again for the above $\varepsilon ~$we have%
\begin{equation}
T(r,A_{m})\leq \exp _{p-1}\left[ \left\{ \phi (r)\right\} ^{\rho
+\varepsilon }\right] ,  \label{2.13}
\end{equation}%
for sufficiently large $r$.

Now, using relations (\ref{2.4}), (\ref{2.5}), (\ref{2.12}), (\ref{2.13})
and choosing $\varepsilon $ such that $0<\varepsilon <\frac{1}{2}\min
\left\{ \rho -\rho _{1},\tau -\tau _{1},\rho _{p}\left( A,\phi \right) -\rho
\right\} ,$ we have from (\ref{2.11}) 
\begin{gather*}
\exp _{p-1}\left[ \left\{ \phi (r)\right\} ^{\rho _{p}\left( A,\phi \right)
-\varepsilon }\right] \leq O\left[ \exp _{p-1}\left[ \left\{ \phi
(r)\right\} ^{\rho _{1}+\varepsilon }\right] \right] +O\left( \exp _{p-1}%
\left[ (\tau _{1}+\varepsilon )\left\{ \phi (r)\right\} ^{\rho }\right]
\right) \\
+\exp _{p-1}\left[ \left\{ \phi (r)\right\} ^{\rho +\varepsilon }\right]
+2(k+1)T(2r,f)+O(1) \\
\Rightarrow T(2r,f)\geq O\left( \exp _{p-1}\left[ \left\{ \phi (r)\right\}
^{\rho _{p}\left( A,\phi \right) -\varepsilon }\right] \right)
\end{gather*}%
for sufficiently large $r$ and $r\in F_{4}.$

Hence we have $i_{\rho }(f,\phi )\geq p~$and $\rho _{p}\left( f,\phi \right)
\geq \rho _{p}\left( A,\phi \right) .$
\end{proof}

\begin{theorem}
\label{T4}Let $A_{j}(z)(j=0,1,\ldots ,k)$ and $A(z)$ be meromorphic
functions and $p=\max \{i_{\rho }(A_{j},\phi ):j=0,1,\ldots ,n-1\}$. If
there exits an $A_{m}(z)(0\leq m\leq k)$ such that 
\begin{eqnarray*}
\lambda _{p}\left( \frac{1}{A_{m}},\phi \right) &<&\mu _{p}\left( A_{m},\phi
\right) <\infty , \\
\max \{\rho _{p}\left( A_{j},\phi \right) &:&j=0,1,\ldots ,k,~j\neq m\}\leq
\mu _{p}\left( A_{m},\phi \right) \\
\text{and }\max \{\tau _{p}\left( A_{j},\phi \right) &:&\rho _{p}\left(
A_{j},\phi \right) =\mu _{p}\left( A_{m},\phi \right) ,~j=0,1,\ldots
,k,~j\neq m\}<t_{p}\left( A_{m},\phi \right) <\infty ,
\end{eqnarray*}%
then the following cases can be arise.

i)\qquad If $\rho _{p}\left( A,\phi \right) <\mu _{p}\left( A_{m},\phi
\right) ,$ or$~\rho _{p}\left( A,\phi \right) =\mu _{p}\left( A_{m},\phi
\right) $ and $\tau _{p}\left( A,\phi \right) <t_{p}\left( A_{m},\phi
\right) ,$ or$~\mu _{p}\left( A,\phi \right) =\mu _{p}\left( A_{m},\phi
\right) $ and $t_{p}\left( A,\phi \right) >t_{p}\left( A_{m},\phi \right) ,$
then every solution $f(z)~(f\neq 0)$ of (\ref{1nh}) satisfies $i_{\rho
}(f,\phi )\geq p~$and $\mu _{p}\left( f,\phi \right) \geq \mu _{p}\left(
A_{m},\phi \right) .$

ii)\qquad If $\mu _{p}\left( A,\phi \right) >\mu _{p}\left( A_{m},\phi
\right) $ then every solution $f(z)~(f\neq 0)$ of (\ref{1nh}) satisfies $%
i_{\rho }(f,\phi )\geq p~$and $\mu _{p}\left( f,\phi \right) \geq \mu
_{p}\left( A,\phi \right) .$
\end{theorem}

\begin{proof}
Let $\varepsilon >0$\ be any arbitrary constant. As in the previous theorem,
from (\ref{2.2})$,$ we get, 
\begin{equation}
T(r,A_{m})\leq \sum_{j=0,j\neq m}^{k}T(r,A_{j})+3\left\{ T(2r,f)\right\}
^{2}+T(r,A)+N(r,A_{m})+O(1)  \label{2a.1}
\end{equation}%
for sufficiently large $r$ and $r\notin F_{1}.$

Now denote,$~\mu =\mu _{p}\left( A_{m},\phi \right) ,~\rho _{1}=\max \{\rho
_{p}\left( A_{j},\phi \right) :j=0,1,\ldots ,k,~\rho _{p}\left( A_{j},\phi
\right) <\mu \},~t=t_{p}\left( A_{m},\phi \right) ,~\tau _{1}=\max \{\tau
_{p}\left( A_{j},\phi \right) :j=0,1,\ldots ,k,~j\neq m,~\rho _{p}\left(
A_{j},\phi \right) =\mu \}$ and $\lambda =\lambda _{p}\left( \frac{1}{A_{m}}%
,\phi \right) .$

Then by the definition of $t_{p}\left( A_{m},\phi \right) ,$ for that $%
\varepsilon >0~$we have%
\begin{equation}
T(r,A_{m})\geq \exp _{p-1}\left[ (t-\varepsilon )\left\{ \phi (r)\right\}
^{\mu }\right]  \label{2a.2}
\end{equation}%
for sufficiently large $r.$

Again if $\rho _{p}\left( A_{j},\phi \right) <\mu ,$ then for that $%
\varepsilon ~$we have%
\begin{equation}
T(r,A_{j})\leq \exp _{p-1}\left[ \left\{ \phi (r)\right\} ^{\mu -\varepsilon
}\right]  \label{2a.3}
\end{equation}%
for sufficiently large $r.$

And if$~\rho _{p}\left( A_{j},\phi \right) =\rho ,~$ $j\neq m,~$then for
that $\varepsilon ~$we have%
\begin{equation}
T(r,A_{j})\leq \exp _{p-1}\left[ (\tau _{1}+\varepsilon )\left\{ \phi
(r)\right\} ^{\mu }\right]  \label{2a.4}
\end{equation}%
for sufficiently large $r.$

Also by the definition of $\lambda _{p}\left( \frac{1}{A_{m}},\phi \right) ,$
we have for the above $\varepsilon $ and for sufficiently large $r$%
\begin{equation}
N(r,A_{m})\leq \exp _{p-1}\left\{ \phi (r)\right\} ^{\lambda +\varepsilon }.
\label{2a.5}
\end{equation}

(i) Case 1: First we take, $\rho _{p}\left( A,\phi \right) <\mu ,$ then for
that $\varepsilon ~$we have%
\begin{equation}
T(r,A)\leq \exp _{p-1}\left[ \left\{ \phi (r)\right\} ^{\rho _{p}\left(
A,\phi \right) +\varepsilon }\right]  \label{2a.6}
\end{equation}%
for sufficiently large $r.$

Now, using the relations (\ref{2a.2})-(\ref{2a.6}) and choosing $\varepsilon 
$ such that $0<\varepsilon <\frac{1}{2}\min \left\{ \mu -\rho _{1},t-\tau
_{1},\mu -\lambda ,\mu -\rho _{p}\left( A,\phi \right) \right\} ,$ we have
from (\ref{2a.1})%
\begin{gather*}
\exp _{p-1}\left[ (\tau -\varepsilon )\left\{ \phi (r)\right\} ^{\mu }\right]
\leq O\left( \exp _{p-1}\left[ \left\{ \phi (r)\right\} ^{\mu -\varepsilon }%
\right] \right) +O\left( \exp _{p-1}\left[ (\tau _{1}+\varepsilon )\left\{
\phi (r)\right\} ^{\mu }\right] \right) \\
+3\left\{ T(2r,f)\right\} ^{2}+\exp _{p-1}\left[ \left\{ \phi (r)\right\}
^{\rho _{p}\left( A,\phi \right) +\varepsilon }\right] +\exp _{p-1}\left[
\left\{ \phi (r)\right\} ^{\lambda +\varepsilon }\right] +O(1) \\
\Rightarrow \left\{ T(2r,f)\right\} ^{2}\geq O\left( \exp _{p-1}\left[
\left\{ \phi (r)\right\} ^{\mu }\right] \right)
\end{gather*}%
for sufficiently large $r$ and $r\notin F_{1}.$

Which implies, $i_{\rho }(f,\phi )\geq p~$and $\mu _{p}\left( f,\phi \right)
\geq \mu =\mu _{p}\left( A_{m},\phi \right) .$

Case 2: Next we suppose that $\rho _{p}\left( A,\phi \right) =\mu $ and $%
\tau _{p}\left( A,\phi \right) <t,~$then for the above $\varepsilon ~$we have%
\begin{equation}
T(r,A)\leq \exp _{p-1}\left[ \left\{ \tau _{p}\left( A,\phi \right)
+\varepsilon \right\} \left\{ \phi (r)\right\} ^{\mu }\right]  \label{2a.7}
\end{equation}%
for sufficiently large $r.$

Now, using relations (\ref{2a.2})-(\ref{2a.5}) and (\ref{2a.7}) and choosing 
$\varepsilon $ such that $0<\varepsilon <\frac{1}{2}\min \left\{ \mu -\rho
_{1},t-\tau _{1},\mu -\lambda ,t-\tau _{p}\left( A,\phi \right) \right\} ,$
we have from (\ref{2a.1})%
\begin{gather*}
\exp _{p-1}\left[ (\tau -\varepsilon )\left\{ \phi (r)\right\} ^{\mu }\right]
\leq O\left( \exp _{p-1}\left[ \left\{ \phi (r)\right\} ^{\mu -\varepsilon }%
\right] \right) +O\left( \exp _{p-1}\left[ (\tau _{1}+\varepsilon )\left\{
\phi (r)\right\} ^{\mu }\right] \right) \\
+3\left\{ T(2r,f)\right\} ^{2}+\exp _{p-1}\left[ \left\{ \tau _{p}\left(
A,\phi \right) +\varepsilon \right\} \left\{ \phi (r)\right\} ^{\mu }\right]
+\exp _{p-1}\left[ \left\{ \phi (r)\right\} ^{\lambda +\varepsilon }\right]
+O(1) \\
\Rightarrow \left\{ T(2r,f)\right\} ^{2}\geq O\left( \exp _{p-1}\left[ (\tau
-\varepsilon )\left\{ \phi (r)\right\} ^{\mu }\right] \right)
\end{gather*}%
for sufficiently large $r$ and $r\notin F_{1}.$

Which implies, $i_{\rho }(f,\phi )\geq p~$and $\mu _{p}\left( f,\phi \right)
\geq \mu =\mu _{p}\left( A_{m},\phi \right) .$

Case 3: Next we take, $\mu _{p}\left( A,\phi \right) =\mu $ and $t_{p}\left(
A,\phi \right) >t,$ then for the above $\varepsilon ~$we have%
\begin{equation}
T(r,A)>\exp _{p-1}\left[ \left\{ t_{p}\left( A,\phi \right) -\varepsilon
\right\} \left\{ \phi (r)\right\} ^{\mu }\right]  \label{2a.8}
\end{equation}%
for sufficiently large $r$.

Again by the definition of $t_{p}\left( A_{m},\phi \right) ,$ and using
Lemma \ref{L4d}, for the above $\varepsilon ~$we have%
\begin{equation}
T(r,A_{m})\leq \exp _{p-1}\left[ (t+\varepsilon )\left\{ \phi (r)\right\}
^{\mu }\right] ,  \label{2a.9}
\end{equation}%
for sufficiently large $r$ with$~r\in F_{2},$ where $lm(F_{2})=+\infty $.

Again as in previous Theorem, from (\ref{2.11}) we get 
\begin{equation}
\Rightarrow T(r,A)\leq \sum_{j=0,j\neq
m}^{k}T(r,A_{j})+T(r,A_{m})+2(k+1)T(2r,f)+O(1)  \label{2a.10}
\end{equation}%
for sufficiently large $r$

Now, using relations (\ref{2a.3}), (\ref{2a.4}), (\ref{2a.8}), (\ref{2a.9})
and choosing $\varepsilon $ such that $0<\varepsilon <\frac{1}{2}\min
\left\{ \mu -\rho _{1},t-\tau _{1},t_{p}\left( A,\phi \right) -t\right\} ,$
we have from (\ref{2a.10}) 
\begin{gather*}
\exp _{p-1}\left[ \left\{ t_{p}\left( A,\phi \right) -\varepsilon \right\}
\left\{ \phi (r)\right\} ^{\mu }\right] \leq O\left( \exp _{p-1}\left[
\left\{ \phi (r)\right\} ^{\mu -\varepsilon }\right] \right) \\
+O\left( \exp _{p-1}\left[ (\tau _{1}+\varepsilon )\left\{ \phi (r)\right\}
^{\mu }\right] \right) +O\left( \exp _{p-1}\left[ (t+\varepsilon )\left\{
\phi (r)\right\} ^{\mu }\right] \right) +(2k+2)T(2r,f)+O(1) \\
\Rightarrow T(2r,f)\geq O\left( \exp _{p-1}\left[ \left\{ t_{p}\left( A,\phi
\right) -\varepsilon \right\} \left\{ \phi (r)\right\} ^{\mu }\right] \right)
\end{gather*}%
for sufficiently large $r$ and $r\in F_{2}.$

It follows that, $i_{\rho }(f,\phi )\geq p~$and $\mu _{p}\left( f,\phi
\right) \geq \mu =\mu _{p}\left( A_{m},\phi \right) .$

(ii) For the second part of the theorem, we take $\mu _{p}\left( A,\phi
\right) >\mu $. Then for the above $\varepsilon ~$we have%
\begin{equation}
T(r,A)\geq \exp _{p-1}\left[ \left\{ \phi (r)\right\} ^{\mu _{p}\left(
A,\phi \right) -\varepsilon }\right]  \label{2a.11}
\end{equation}

Again by using Lemma \ref{L4c}, for the above $\varepsilon ~$we have%
\begin{equation}
T(r,A_{m})\leq \exp _{p-1}\left[ \left\{ \phi (r)\right\} ^{\mu +\varepsilon
}\right] ,  \label{2a.12}
\end{equation}%
for sufficiently large $r$ with$~r\in F_{3},$ where $lm(F_{3})=+\infty $.

Now, using relations (\ref{2a.3}), (\ref{2a.4}), (\ref{2a.11}) and (\ref%
{2a.12}) and choosing $\varepsilon $ such that $0<\varepsilon <\frac{1}{2}%
\min \left\{ \mu -\rho _{1},t-\tau _{1},\mu _{p}\left( A,\phi \right) -\mu
\right\} ,$ we have from (\ref{2a.10}) 
\begin{gather*}
\exp _{p-1}\left[ \left\{ \phi (r)\right\} ^{\mu _{p}\left( A,\phi \right)
-\varepsilon }\right] \leq O\left( \exp _{p-1}\left[ \left\{ \phi
(r)\right\} ^{\mu -\varepsilon }\right] \right) +O\left( \exp _{p-1}\left[
(\tau _{1}+\varepsilon )\left\{ \phi (r)\right\} ^{\mu }\right] \right) \\
+\exp _{p-1}\left[ \left\{ \phi (r)\right\} ^{\mu +\varepsilon }\right]
+(2k+2)T(2r,f)+O(1) \\
\Rightarrow T(2r,f)\geq O\left( \exp _{p-1}\left[ \left\{ \phi (r)\right\}
^{\mu _{p}\left( A,\phi \right) -\varepsilon }\right] \right)
\end{gather*}%
for sufficiently large $r$ and $r\in F_{3}.$

It follows that, $i_{\rho }(f,\phi )\geq p~$and $\mu _{p}\left( f,\phi
\right) \geq \mu _{p}\left( A,\phi \right) .$

This completes the proof.
\end{proof}

Next we state two corollaries considering the homogeneous linear difference
equation (\ref{1h}) which directly follows from Theorem \ref{T3} and Theorem %
\ref{T4}.

\begin{corollary}
Let $A_{j}(z)(j=0,1,\ldots ,k)$ be meromorphic coefficients of (\ref{1h})
and $p=\max \{i_{\rho }(A_{j},\phi ):j=0,1,\ldots ,n-1\}$. If there exits an 
$A_{m}(z)(0\leq m\leq k)$ such that 
\begin{eqnarray*}
\lambda _{p}\left( \frac{1}{A_{m}},\phi \right) &<&\rho _{p}\left(
A_{m},\phi \right) <\infty , \\
\max \{\rho _{p}\left( A_{j},\phi \right) &:&j=0,1,\ldots ,k,~j\neq m\}\leq
\rho _{p}\left( A_{m},\phi \right) \\
\text{and }\max \{\tau _{p}\left( A_{j},\phi \right) &:&\rho _{p}\left(
A_{j},\phi \right) =\rho _{p}\left( A_{m},\phi \right) ,~j=0,1,\ldots
,k,~j\neq m\}<\tau _{p}\left( A_{m},\phi \right) <\infty ,
\end{eqnarray*}%
then every solution $f(z)~(f\neq 0)$ of (\ref{1h}) satisfies $i_{\rho
}(f,\phi )\geq p~$and $\rho _{p}\left( f,\phi \right) \geq \rho _{p}\left(
A_{m},\phi \right) .$
\end{corollary}

\begin{proof}
Putting $A=0$ in the Theorem \ref{T3} we get the required result.
\end{proof}

\begin{corollary}
Let $A_{j}(z)(j=0,1,\ldots ,k)$ and $A(z)$ be meromorphic functions and $%
p=\max \{i_{\rho }(A_{j},\phi ):j=0,1,\ldots ,n-1\}$. If there exits an $%
A_{m}(z)(0\leq m\leq k)$ such that 
\begin{eqnarray*}
\lambda _{p}\left( \frac{1}{A_{m}},\phi \right) &<&\mu _{p}\left( A_{m},\phi
\right) <\infty , \\
\max \{\rho _{p}\left( A_{j},\phi \right) &:&j=0,1,\ldots ,k,~j\neq m\}\leq
\mu _{p}\left( A_{m},\phi \right) \\
\text{and }\max \{\tau _{p}\left( A_{j},\phi \right) &:&\rho _{p}\left(
A_{j},\phi \right) =\mu _{p}\left( A_{m},\phi \right) ,~j=0,1,\ldots
,k,~j\neq m\}<t_{p}\left( A_{m},\phi \right) <\infty ,
\end{eqnarray*}%
then every solution $f(z)~(f\neq 0)$ of (\ref{1h}) satisfies $i_{\rho
}(f,\phi )\geq p~$and $\mu _{p}\left( f,\phi \right) \geq \mu _{p}\left(
A_{m},\phi \right) .$
\end{corollary}

\begin{proof}
Putting $A=0$ in the Theorem \ref{T4} we get the required result.
\end{proof}

Our Theorem \ref{T5} and Theorem \ref{T6} are improvements of the theorems
due to Bela\"{\i}di-Benkarouba (\cite{bel}, Theorem 1.1 and Theorem 1.2).
Here we modify the hypothesis by giving some conditions on the
Characteristic functions of the coefficients of (\ref{1h}) rather than
taking coefficient having maximal iterated $p-\phi $ order.

\begin{theorem}
\label{T5}Let $A_{0}(z),A_{1}(z),....,A_{k}(z)$ be meromorphic coefficients
of (\ref{1h}) satisfying $\max \{i_{\rho }(A_{j},\phi ):j=0,1,\ldots
,n-1\}=p $ and $\max \left\{ \rho _{p}\left( A_{j},\phi \right) :0\leq j\leq
k\right\} \leq \rho .$ If there exits an $A_{m}(z)(0\leq m\leq k)$ with $%
\lambda _{p}\left( \frac{1}{A_{m}},\phi \right) <\rho _{p}\left( A_{m},\phi
\right) $ and a set $F\subset (1,+\infty )$ with $lm(F)=+\infty $ such that
for some constants $a,b$ $(0\leq b<a)$ and $\delta $ $\left( 0<\delta <\rho
\right) $ sufficiently small with 
\begin{eqnarray}
T(r,A_{m}) &\geq &\exp _{p-1}\left[ a\left\{ \phi \left( r\right) \right\}
^{\rho -\delta }\right] ,  \label{3.1} \\
\text{and }T(r,A_{j}) &\leq &\exp _{p-1}\left[ b\left\{ \phi \left( r\right)
\right\} ^{\rho -\delta }\right] ,\text{ }j=0,1,...,k;\text{ }j\neq l,
\label{3.2}
\end{eqnarray}%
as $r\rightarrow \infty $ for $r\in F,$ then every solution $f(z)~(f\neq 0)$
of (\ref{1h}) satisfies $i_{\rho }(f,\phi )\geq p~$and $\rho _{p}\left(
f,\phi \right) \geq \rho _{p}\left( A_{m},\phi \right) .$
\end{theorem}

\begin{proof}
First we shall show that $\rho _{p}\left( A_{m},\phi \right) =\rho .~$If not
let us assume that $\rho >\rho _{p}\left( A_{m},\phi \right) =\sigma $
(say). Then for given any $\varepsilon >0$ and for sufficiently large $r,$we
have%
\begin{equation*}
T(r,A_{m})\leq \exp _{p-1}\left[ \left\{ \phi \left( r\right) \right\}
^{\sigma +\varepsilon }\right]
\end{equation*}

Again by (\ref{3.1}), there exists $\delta \left( 0<\delta <\rho \right) ,$
such that 
\begin{equation*}
T(r,A_{m})\geq \exp _{p-1}\left[ a\left\{ \phi \left( r\right) \right\}
^{\rho -\delta }\right]
\end{equation*}%
as $r\rightarrow \infty $ for $r\in F.$

Now choosing $\varepsilon <\rho -\sigma -2\delta $ and combining the above
two inequalities, we obtain for $r\in F$%
\begin{equation*}
\exp _{p-1}\left[ a\left\{ \phi \left( r\right) \right\} ^{\rho -\delta }%
\right] \leq T(r,A_{m})\leq \exp _{p-1}\left[ \left\{ \phi \left( r\right)
\right\} ^{\sigma +\varepsilon }\right] ,
\end{equation*}%
which is a contradiction as $r\rightarrow \infty .$

Hence $\rho _{p}\left( A_{m},\phi \right) =\rho .\qquad $

Now let $f$ be a nonzero meromorphic solution of (\ref{1h}) and put $c_{0}=0$%
. If possible let us assume that $\rho >\rho _{p}\left( f,\phi \right)
=\kappa ~$(say). Then by Lemma \ref{L5a}, for given any $\varepsilon >0$ and
for sufficiently large $r,$ we have%
\begin{equation}
m\left( r,\frac{f(z+c_{j})}{f(z+c_{m})}\right) =O\left( \exp _{p-1}\left[
\left\{ \phi \left( r\right) \right\} ^{\kappa +\varepsilon }\right] \right)
,~\text{for }j\neq m.  \label{3.3}
\end{equation}

Also by the definition of $\lambda _{p}\left( \frac{1}{A_{m}},\phi \right)
=\lambda ,$ (say) we have for the above $\varepsilon $ and for sufficiently
large $r$%
\begin{equation}
N(r,A_{m})\leq \exp _{p-1}\left\{ \phi (r)\right\} ^{\lambda +\varepsilon }.
\label{3.4}
\end{equation}

By using similar arguments as in Theorem 1, from (\ref{1.1a}) we get,

\begin{eqnarray}
&&T(r,A_{m})=m(r,A_{m})+N(r,A_{m})  \notag \\
&\leq &\sum_{j=0,j\neq m}^{k}m(r,A_{j})+\sum_{j=0,j\neq m}^{k}m\left( r,%
\frac{f(z+c_{j})}{f(z+c_{m})}\right) +N(r,A_{m})+O(1)  \label{3.5}
\end{eqnarray}%
for$~r\notin F_{1},$ where $lm(F_{1})<+\infty $

Now, using the relations (\ref{3.1})-(\ref{3.4}) and choosing $\varepsilon $
such that $0<\varepsilon <\frac{1}{2}\min \left\{ \rho -\kappa -2\delta
,\rho -\lambda -2\delta \right\} ,$ we have from (\ref{3.5})%
\begin{eqnarray*}
\exp _{p-1}\left[ \left\{ a\phi \left( r\right) \right\} ^{\rho -\delta }%
\right] &\leq &k\exp _{p-1}\left[ b\left\{ \phi \left( r\right) \right\}
^{\rho -\delta }\right] \\
&&+\exp _{p-1}\left[ \left\{ \phi \left( r\right) \right\} ^{\kappa
+\varepsilon }\right] +\exp _{p-1}\left[ \left\{ \phi (r)\right\} ^{\lambda
+\varepsilon }\right] +O(1).
\end{eqnarray*}%
for $r\in F\backslash F_{1}.$

Which implies%
\begin{equation*}
(a-b)\left\{ \phi \left( r\right) \right\} ^{\rho -\delta }\leq \left\{ \phi
\left( r\right) \right\} ^{\kappa +\varepsilon }+\left\{ \phi (r)\right\}
^{\lambda +\varepsilon }+O(1),\text{ for }r\in F\backslash F_{1}.
\end{equation*}

Since $(a-b)>0,$ the above expression implies%
\begin{equation*}
1\leq \frac{\left\{ \phi \left( r\right) \right\} ^{\kappa +\varepsilon
-\rho +\delta }+\left\{ \phi (r)\right\} ^{\lambda +\varepsilon -\rho
+\delta }}{(a-b)}+\frac{O(1)}{(a-b)\left\{ \phi \left( r\right) \right\}
^{\rho -\delta }}\rightarrow 0\text{ as }r\rightarrow +\infty ,
\end{equation*}%
which is a contradiction.

Thus we have, $\rho _{p}\left( f,\phi \right) \geq \rho =\rho _{p}\left(
A_{m},\phi \right) $ as well as $i_{\rho }(f,\phi )\geq p.$
\end{proof}

Our next theorem i.e. Theorem \ref{T6} is the entire version of the previous
theorem.

\begin{theorem}
\label{T6}Let $A_{0}(z),A_{1}(z),....,A_{k}(z)$ be entire coefficients of (%
\ref{1h}) satisfying $\max \{i_{\rho }(A_{j},\phi ):j=0,1,\ldots ,n-1\}=p$
and $\max \left\{ \rho _{p}\left( A_{j},\phi \right) :0\leq j\leq k\right\}
\leq \rho .$If there exits an $A_{m}(z)(0\leq m\leq k)$ and a set $F\subset
(1,+\infty )$ with $lm(F)=+\infty $ such that for some constants $a,b$ $%
(0\leq b<a)$ and $\delta $ $\left( 0<\delta <\rho \right) $ sufficiently
small with 
\begin{eqnarray}
\left\vert A_{m}(z)\right\vert &\geq &\exp _{p}\left[ a\left\{ \phi \left(
r\right) \right\} ^{\rho -\delta }\right]  \label{4.1} \\
\text{and }\left\vert A_{j}(z)\right\vert &\leq &\exp _{p}\left[ b\left\{
\phi \left( r\right) \right\} ^{\rho -\delta }\right] ,\text{ }j=0,1,...,k;%
\text{ }j\neq l,  \label{4.2}
\end{eqnarray}%
as $r\rightarrow \infty $ for $r\in F,$ then every solution $f(z)~(f\neq 0)$
of (\ref{1h}) satisfies $i_{\rho }(f,\phi )\geq p~$and $\rho _{p}\left(
f,\phi \right) \geq \rho _{p}\left( A_{m},\phi \right) .$
\end{theorem}

\begin{proof}
First we shall show that $\rho _{p}\left( A_{m},\phi \right) =\rho .~$If not
let us assume that $\rho >\rho _{p}\left( A_{m},\phi \right) =\sigma $
(say). Then for given any $\varepsilon >0$ and for sufficiently large $r,$we
have%
\begin{equation*}
\left\vert A_{m}(z)\right\vert \leq \exp _{p}\left[ \left\{ \phi \left(
r\right) \right\} ^{\sigma +\varepsilon }\right]
\end{equation*}

Again by (\ref{4.1}), there exists $\delta \left( 0<\delta <\rho \right) ,$
such that 
\begin{equation*}
\left\vert A_{m}(z)\right\vert \geq \exp _{p}\left[ a\left\{ \phi \left(
r\right) \right\} ^{\rho -\delta }\right]
\end{equation*}%
as $r\rightarrow \infty $ for $r\in F.$

Now choosing $\varepsilon <\rho -\sigma -2\delta $ and combining the above
two inequalities, we obtain for $r\in F$%
\begin{equation*}
\exp _{p}\left[ a\left\{ \phi \left( r\right) \right\} ^{\rho -\delta }%
\right] \leq \left\vert A_{m}(z)\right\vert \leq \exp _{p}\left[ \left\{
\phi \left( r\right) \right\} ^{\sigma +\varepsilon }\right] ,
\end{equation*}%
which is a contradiction as $r\rightarrow \infty .$

Hence $\rho _{p}\left( A_{m},\phi \right) =\rho .$

Now let $f$ be a nonzero meromorphic solution of (\ref{1h}) and put $c_{0}=0 
$. If possible let us assume that $\rho >\rho _{p}\left( f,\phi \right)
=\kappa ~$(say).

Now divide (\ref{1h}) by $f(z+c_{m})$ we get%
\begin{gather}
-A_{m}(z)=\sum_{j=0,j\neq m}^{k}A_{j}(z)\frac{f(z+c_{j})}{f(z+c_{m})}. 
\notag \\
\Rightarrow -1=\sum\limits_{j=0,j\neq l}^{k}\frac{A_{j}(z)f(z+c_{j})}{%
A_{m}(z)f(z+c_{m})}  \notag \\
\Rightarrow 1\leq \sum\limits_{j=0,j\neq l}^{k}\left\vert \frac{%
A_{j}(z)f(z+c_{j})}{A_{m}(z)f(z+c_{m})}\right\vert  \label{4.3}
\end{gather}

Choose $\varepsilon >0$ such that $0<\varepsilon <\frac{1}{2}\min \left\{
\rho -\kappa -2\delta ,\rho -\lambda -2\delta \right\} .$ Then by Lemma \ref%
{L5c}, for that $\varepsilon $ there exists a subset $F_{1}\subset \left(
1,+\infty \right) $ with $lm(F_{1})<+\infty $ such that for all $r\notin
F_{1}\cup \left[ 0,1\right] ,$ we have%
\begin{equation}
\left\vert \frac{f(z+c_{j})}{f(z+c_{m})}\right\vert \leq \exp _{p}\left[
\left\{ \phi \left( r\right) \right\} ^{\kappa +\varepsilon }\right] <\exp
_{p}\left[ \left\{ \phi \left( r\right) \right\} ^{\rho -2\delta }\right] ,~%
\text{for }j\neq m.  \label{4.4}
\end{equation}

Thus for $r\in F\backslash \left( F_{1}\cup \left[ 0,1\right] \right) ,$
substituting (\ref{4.1}), (\ref{4.2}) and (\ref{4.4}) in (\ref{4.3}), we
obtain%
\begin{equation*}
1\leq k\frac{\exp _{p}\left[ b\left\{ \phi \left( r\right) \right\} ^{\rho
-\delta }\right] }{\exp _{p}\left[ a\left\{ \phi \left( r\right) \right\}
^{\rho -\delta }\right] }\exp _{p}\left[ \left\{ \phi \left( r\right)
\right\} ^{\rho -2\delta }\right] \rightarrow 0\text{ as }r\rightarrow
\infty .
\end{equation*}

Which is a contradiction.

Hence we get $\rho _{p}\left( f,\phi \right) \geq \rho =\rho _{p}\left(
A_{m},\phi \right) $ as well as $i_{\rho }(f,\phi )\geq p.$
\end{proof}


\begin{thebibliography}{99}
\bibitem{bel} B. Bela\"{\i}di and Y. Benkarouba, Some properties of
meromorphic solutions of higher order linear difference equations, Ser. A
Appl. Math. Inform. Mech. 11(2), (2019) 75--95.

\bibitem{bel2} B. Bela\"{\i}di and R. Bellaama, Study of the growth
properties of meromorphic solutions of higher-order linear difference
equations, Arab. J. Math. (2021). https://doi.org/10.1007/s40065-021-00324-2.

\bibitem{cao} T. B. Cao, J. F. Xu and Z. X. Chen, On the meromorphic
solutions of linear differential equations on the complex plane. J. Math.
Anal. Appl. 364(1), (2010) 130--142.

\bibitem{chi} Y. M. Chiang and S. J. Feng, On the Nevanlinna characteristic
of $f\left( z+\eta \right) $ and difference equations in the complex plane,
Ramanujan J. 16(1), (2008) 105--129.

\bibitem{chy} I. Chyzhykov, J. Heittokangas and J. R\"{a}tty\"{a}:
Finiteness of $\phi $-order of solutions of linear differential equations in
the unit disc. J. Anal. Math., Vol. 109, (2009) 163-198.

\bibitem{gol} A. Goldberg and I. Ostrovskii, Value Distribution of
Meromorphic functions, Transl. Math. Monogr., vol. 236, Amer. Math. Soc.,
Providence RI, 2008.

\bibitem{hal} R. G. Halburd and R. J. Korhonen, Difference analogue of the
lemma on the logarithmic derivative with applications to difference
equations, J. Math. Anal. Appl. 314(2), (2006) 477--487.

\bibitem{hay} W. K. Hayman, Meromorphic Functions, Oxford Mathematical
Monographs. Clarendon Press, Oxford, 1964.

\bibitem{hu} H. Hu and X. M. Zheng, Growth of solutions to linear
differential equations with entire coefficients. Electron. J. Differ. Equ.
2012 (226), (2012) 1-15.

\bibitem{kin} L. Kinnunen, Linear differential equations with solutions of
finite iterated order. Southeast Asian Bull. Math. 22 (4) (1998) 385--405.

\bibitem{lai} I. Laine and C. C. Yang, Clunie theorems for difference and $%
q- $difference polynomials. J. Lond. Math. Soc. (2) 76(3), (2007) 556--566.

\bibitem{lat} Z. Latreuch and B, Bela\"{\i}di, Growth and oscillation of
meromorphic solutions of linear difference equations.Mat. Vesnik 66(2),
(2014) 213--222.

\bibitem{liu} H. Liu and Z. Mao, On the meromorphic solutions of some linear
difference equations. Adv. Differ. Equ. 2013(133), (2013) 1--12.

\bibitem{she} X. Shen, J. Tu and H. Y. Xu: Shen et al.: Complex oscillation
of a second-order linear differential equation with entire coefficients of $%
[p,q]-\phi $ order. Advances in Difference Equations 2014, 2014:200.

\bibitem{zhe} X. M. Zheng and J. Tu, Growth of meromorphic solutions of
linear difference equations, J. Math. Anal. Appl. 384 (2), (2011) 349--356.

\bibitem{zho} Y. P. Zhou and X. M. Zheng, Growth of meromorphic solutions to
homogeneous and non-homogeneous linear (differential-)difference equations
with meromorphic coefficients. Electron. J. Differ. Equ. 2017(34), (2017)
1-15.
\end{thebibliography}
\end{document}